\documentclass[letterpaper, 10 pt, conference]{ieeeconf}

\IEEEoverridecommandlockouts    
\usepackage{amsthm}

\usepackage{amsmath}
\usepackage{amssymb,amsfonts,amscd}
\usepackage{pifont}

\usepackage{graphicx}
\usepackage[table]{xcolor}
\usepackage{float}   
\usepackage{array}   
\usepackage{booktabs}  
\usepackage{caption}

\newtheorem{theorem}{Theorem}
\newtheorem{lemma}[theorem]{Lemma}

\newtheorem{proposition}[theorem]{Proposition}
\newtheorem{remark}[theorem]{Remark}

\definecolor{verylightgray}{gray}{0.95}
\newcolumntype{L}{>{$}l<{$}} 
\newcolumntype{G}{>{\columncolor{verylightgray}$}l<{$}}

\def\R{\mathbb{R}}

\def\AA{\mathcal{A}}
\def\DD{\mathcal{D}}
\def\RR{\mathcal{R}}
\def\UU{\mathcal{U}}

\def\phi{\varphi}

\def\Sin{s_{\text{in}}}

\def\Vin{v_{\text{in}}}

\def\dmax{d_{\max}}
\def\mumax{\mu_{\max}}
\def\rhomax{\rho_{\max}}
\def\phimax{\varphi_{\max}}
\def\psimax{\psi_{\max}}
\def\qmax{q_{\max}}
\def\qmin{q_{\min}}

\def\asing{\alpha_{\mathrm{singular}}}
\def\dsing{d_{\mathrm{singular}}}

\def\xx{\mathbf{x}}
\def\xeq{\xx^*}
\def\xa{\xx^A}
\def\xb{\xx^B}
\def\xoa{\xx_0^A}
\def\xob{\xx_0^B}
\def\xia{\xx_1^A}
\def\xib{\xx_1^B}
\def\xo{\xx_0}
\def\xio{\xx_{1,0}}
\def\xii{\xx_{1,1}}
\def\dxt{\dot{\xx}}
\def\dxa{\dxt^A}
\def\dxb{\dxt^B}

\def\uu{\mathbf{u}}
\def\vv{\mathbf{v}}

\def\lmd{{\boldsymbol\lambda}}
\def\ls{\lambda_s}
\def\le{\lambda_e}
\def\lv{\lambda_v}
\def\lc{\lambda_c}
\def\lq{\lambda_q}

\def\dlmd{\dot{\lmd}}

\def\dt{\text{d}t}

\def\pp#1{\left(#1\right)}
\def\bb#1{\left[#1\right]}
\def\cb#1{\left\{#1\right\}}

\def\dotp#1{\left\langle#1\right\rangle}

\def\transp#1{{#1}^{\top}}

\def\L{L}
\def\gL{g{\cdot}\L^{-1}}

\def\perday{\text{day}^{-1}}

\def\gg{g{\cdot}g^{-1}}
\def\ggday{\gg{\cdot}\perday}

\title{\LARGE \bf
Optimization of microalgae biosynthesis via controlled algal-bacterial symbiosis* 
}

\author{R. Asswad$^1$, W. Djema$^{2}$, O. Bernard$^{2}$, J.-L. Gouzé$^{3,\dag}$, E. Cinquemani$^{1,\dag}$
\thanks{*Work 
supported in part by project Ctrl-AB [ANR-20-CE45-0014]}
\thanks{$^\dag$These authors contributed equally to the work}
\thanks{$^1$MICROCOSME, Centre Inria de l'Université Grenoble Alpes
        {\tt\small rand.asswad@inria.fr} ; {\tt\small eugenio.cinquemani@inria.fr}}%
\thanks{$^2$BIOCORE, Centre Inria d'Université Côte d'Azur
        {\tt\small walid.djema@inria.fr} ; {\tt\small olivier.bernard@inria.fr}}%
\thanks{$^3$MACBES, Centre Inria d'Université Côte d'Azur
        {\tt\small jean-luc.gouze@inria.fr}
}%
}

\begin{document}

\maketitle
\thispagestyle{empty}
\pagestyle{empty}

\begin{abstract}

We investigate optimization of an algal-bacterial consortium, where 
an exogenous control input modulates bacterial resource allocation between growth and synthesis of a resource that is limiting for algal growth.
Maximization of algal biomass synthesis is pursued in a continuous bioreactor, with dilution rate as an additional control variable. We formulate optimal control in the two variants of static and dynamic control problems, and address them by theoretical and numerical tools. We explore convexity of the static problem and uniqueness of its solution, and show that the dynamic problem displays a solution with bang-bang control actions and singular arcs that result in cyclic control actions. We finally discuss the relation among the two solutions and show the extent to which dynamic control can outperform static optimal solutions.
\end{abstract}

\section{Introduction}\label{intro}

Microbial consortia, that is, communities of several microbial species in interaction, are common
in nature as the result of co-evolution \cite{sachs2011evolutionary}. Such natural symbioses are increasingly exploited to reduce the needs in fertilizers, vitamins or micro nutriments, by having these elements produced within the consortium instead of adding them artificially \cite{rapp2020partners}.
A chief example is that of mixed cultures of algae and bacteria. The potential of microalgae for the production of food, feed, cosmetics or pharmaceutics molecules has been highlighted in the last decade \cite{rizwan2018exploring}. It has been shown that co-culturing algae with probiotic bacteria could lead to a cost decrease, higher robustness of the culture and even growth enhancement \cite{palacios2022microalga}. 

The drawback of using several microorganisms in place of single species is the increased complexity of the resulting dynamics,
with challenging issues for process control and optimization~\cite{zomorrodi2016synthetic}.
State-of-the-art modelling and experimental capabilities sharpened the mathematical characterization of bioprocess dynamics \cite{lopatkin2020predictive}. While a range of methods for online estimation, optimization and control have been developed over the past decades \cite{simutis2015bioreactor}, they are often based on models of the growth of a single species in a bioreactor. Control methods dedicated to microbial consortia are still scarce~\cite{aditya2021light,salzano2022ratiometric} 
and much remains to be explored.

In this paper, we focus on the optimal control of an algal-bacterial consortium growing in a continuous bioreactor. We consider the scenario where algal growth is dependent on a limiting resource, such as vitamins, synthesized by a bacterial population (see also~\cite{mauri2020enhanced} and references therein for consortia displaying similar interactions).
Bacterial synthesis of the vitamins is regulated by an exogenous control input (\textit{e.g.} by the use of optogenetics~\cite{benisch2024unlocking}), and we make the realistic assumption that, as in analogous resource re-allocation scenarios~\cite{WeisseEtAl2015}, increased synthesis comes at the expense of reduced bacterial growth. Motivated by a synthetic consortium under development comprising microalga \textit{Chlorella} and bacterium \textit{Escherichia coli}, this provides a rich scenario to study control of microbial consortia and maximization of productivity of algal biomass, a proxy for optimized biosynthesis of added value compounds. 

We consider two alternative Optimal Control Problems (OCPs): a Static OCP (SOCP), where productivity is to be maximized at equilibrium relative to dilution and vitamin synthesis rate parameters, and a Dynamic OCP (DOCP), where total productivity is to be maximized over a finite time period with time-varying rate functions. 
For the SOCP, we will first explore equilibria and stability of the algal-bacterial consortium, notably in the conditions for species coexistence. We then investigate problem convexity and, in the same spirit as~\cite{mauri2020enhanced}, we numerically illustrate the control tradeoffs and the uniqueness of the solution. For the DOCP, we seek the solution via geometric control theory (\cite{schattler2012geometric,vinter2010optimal}) and the Pontryagin's Maximum Principle (PMP, see \textit{e.g.}, \cite{schattler2012geometric,pontryagin2018mathematical}  
and related applications in~\cite{djema2022optimal,yegorov2019optimal,giordano2016dynamical}). We establish a solution in the form of bang-bang control actions and singular arcs.
Then, on a simulated case study,
we use the numerical optimization software {\tt Bocop}~\cite{Bocop} to illustrate the theoretical results, extend them by showing in particular the cyclic nature of the solution, and relate the solution with that of the corresponding SOCP.

The paper is organized as follows. In Section \ref{model}
we introduce and analyze the algal-bacterial consortium model.  The OCPs of interest are stated in Section \ref{pb}. Section \ref{socp} is dedicated to the
analysis and numerical solution of the SOCP, while Section
\ref{docp} focuses on the theoretical solution of the DOCP.
Simulation results comparing the solutions of the SOCP and
of the DOCP are developed in Section \ref{results}.
Conclusions and perspectives are in Section \ref{conclusion}. 

\section{Microbial consortium model and analysis}\label{model}

We present a model that describes the co-culturing of \textit{E. coli} bacteria and \textit{Chlorella} microalgae in a continuously stirred-tank bioreactor. Bacterial growth dynamics are described by the well-known Monod model under glucose limitation~\cite{bastin_1990}. Bacteria are engineered so as to synthesize a vitamin that is a growth-limiting resource for the algal strain. Vitamin synthesis is controlled through optogenetics \cite{benisch2024unlocking}, and it is assumed that allocation of resources to vitamin synthesis slows down bacterial growth~(\cite{WeisseEtAl2015} and own experimental evidence). 
Algal growth dynamics are described by the Caperon-Droop model \cite{smith_1994}, treating vitamins as the limiting growth substrate. Also known as variable yield model, this model accounts for the import of substrate contributing to an internal quota and subsequent utilization of this quota for algal biomass synthesis. 

Let $s$, $e$, $v$, and $c$ be the glucose, bacterial biomass,
secreted vitamin and algal biomass concentration in the
(fixed volume) bioreactor, in the same order. Let $q$ be the
internal algal vitamin quota. Defining the system state
$\xx(t)=\transp{(s(t),e(t),v(t),q(t),c(t))}$,
its dynamics are given by
\begin{align}
\dot{s} &= -\frac{1}{\gamma}\phi(s)e + d(t)(\Sin - s)\label{eq:ds}\\
\dot{e} &= (1 - \alpha(t))\phi(s)e - d(t)e\label{eq:de}\\
\dot{v} &= \alpha(t)\beta\phi(s)e - \rho(v)c - d(t)v\label{eq:dv}\\
\dot{q} &= \rho(v) - \mu(q)q\label{eq:dq}\\
\dot{c} &= \mu(q)c - d(t)c\label{eq:dc}
\end{align}
with functions $\phi$ and $\rho$ defined over $[0,\infty)$,
and $\mu$ defined over $[\qmin,\infty)$, given by
\begin{equation*}
\phi(s)=\frac{\phimax s}{k_s+s},
\rho(v)=\frac{\rhomax v}{k_v+v},
\mu(q)=\mumax\pp{1-\frac{\qmin}{q}}.
\end{equation*}
State $\xx$ takes values in the space $\Omega=\R_+\setminus\cb{q<\qmin}$.

We consider throughout the article a constant substrate input flow $\Sin$.
Constant $\gamma$ represents the bacterial growth yield,
$\phi(s)$ and $\mu(q)$ are the specific growth rates
of bacteria and microalgae, respectively. $\rho(v)$ is the vitamin uptake
rate, $d(t)$ is the dilution rate and $\beta$ is the vitamin production yield.
All the model parameters are strictly positive.
Functions $\phi$ and $\rho$ are called Monod functions as they have the form
$t\mapsto at/(b+t)$ defined on $\R_+$ with $a$ and $b$ strictly positive.
Factors $1-\alpha(t)$ and $\alpha(t)$ in Eq.~\eqref{eq:de}-\eqref{eq:dv} represent the allocation of resources toward biomass or vitamin synthesis in response to optogenetic control $\alpha(t)$. For simplicity we do not model the optogenetic system response dynamics 
to light induction (the actual exogenous control input). We simply let $\alpha(t)$ be the control variable and assume we can set it to any value in $[0,1]$ at any time $t$.

For numerical simulations only, we will 
refer to the specific choice of parameter values in Table~\ref{tab:jc_params} (ratios of similar units, \textit{e.g.} $\gg$, refer to distinct molecules, \textit{e.g.} grams of vitamins per gram of biomass). Parameter values pertaining the algal dynamics ($\rhomax$, $\qmin$ and $\mumax$) were determined based on experimental data obtained with collaborators at the Laboratoire d'Océanographie de Villefranche (IMEV, Villefranche-sur-Mer, France). Bacterial growth parameters $\gamma$, 
$\phimax$ and $k_s$ were borrowed from~\cite{mauri2020enhanced}, with $\phimax$ reduced so as to account for the growth of bacteria at a lower-than-optimal temperature (well below 37$^\circ$ for optimal algal growth). Yield $\beta$ was borrowed from~\cite{lin_2014}. 

\begin{table}[h]
    \centering
    \caption{Model parameters}
    \label{tab:jc_params}
    \begin{tabular}{|G|LL|G|LL|}
        \hline
        k_v & 0.57 & m\gL &
        k_s & 0.09& \gL \\
        \rhomax & 27.3 & m\ggday&
        \phimax & 6.48 & \perday\\
        \qmin & 2.7628 &  m\gg &
        \gamma & 0.44 & \gg\\
        \mumax & 1.0211 & \perday &
        \beta & 23 & m\gg\\
        \hline
    \end{tabular}
\end{table}

\subsection{Steady states and asymptotic behavior}

The system can be studied in two parts: the dynamics of bacterial growth
$\xb(t)=\transp{(s(t),e(t))}\in\Omega_B=\R_+^2$, and those of algal growth
$\xa(t)=\transp{(v(t),q(t),c(t))}\in\Omega_A=\R_+^3\setminus\cb{q<\qmin}$.
While both systems have been studied separately
\cite{bastin_1990,smith_1994,lange_1992}, in this paper a novel consortium
system is presented and studied.
The model is a cascade autonomous dynamical system
\begin{align}
\dxb(t) &= f_B(\xb(t))\label{eq:bac}\\
\dxa(t) &= f_A(\xb(t), \xa(t))\label{eq:alg}
\end{align}
where the bacterial dynamics (\ref{eq:bac}) do not depend on $\xa(t)$.

\subsubsection{Bacterial steady states}
The system (\ref{eq:bac}) has a trivial equilibrium point
$\xob=\transp{(\Sin, 0)}$ referred to as the bacterial washout equilibrium.
A nontrivial equilibrium exists $\xib=\transp{(s^*,e^*)}$
if $d<(1-\alpha)\phi(\Sin)$ with
\begin{equation}\label{eq:Xb-coord}
s^*=\phi^{-1}\pp{\frac{d}{1-\alpha}}~\text{and}~
e^*=(1-\alpha)\gamma(\Sin - s^*).
\end{equation}
If $\xib$ exists, it is globally asymptotically stable (GAS) over
$\Omega_B\setminus\cb{e=0}$ and the washout is unstable,
otherwise the latter is GAS over $\Omega_B$ \cite{harmand_2007}.

Let $\Vin(\xb) = \alpha\beta\phi(s)e/d$, it follows that $\Vin=0$
at bacterial washout and if $\xib$ exists then
\begin{equation}\label{eq:vin}
\Vin^*=\Vin(\xib)=\alpha\beta\gamma(\Sin-s^*)>0.
\end{equation}
This allows reformulating the system (\ref{eq:alg}) in the usual
form of a Droop model.

\subsubsection{Algal steady states}

Algal equilibria can be investigated by exploring the zero
dynamics of \eqref{eq:alg}:
{\small
\begin{align}
0 &= -\rho(v)c + d(\Vin - v) & &\implies c=d(\Vin - v)/\rho(v) \label{eq:Iv}\\
0 &= \rho(v) - \mu(q)q & &\implies \rho(v) = \mu(q)q \label{eq:Iq}\\
0 &= \mu(q)c - dc & &\implies c=0\text{ or }\mu(q) = d \label{eq:Ic}
\end{align}
}

\ding{111} The case of $\Vin=0$ leads to a degenerate Droop model
that has a single equilibrium at $\transp{(0,\qmin,0)}$.

\ding{111} In the nondegenerate case of $\Vin>0$, the algal system
has up to two distinct equilibria that we explore at the bacterial
equilibrium $\xib$ for which the vitamin feed is $\Vin^*$.

{\scriptsize\ding{117}}  If $c=0$ the system is at an algal washout equilibrium
$\xoa=\transp{(\Vin^*,q_0,0)}$ with $q_0=\qmin+\rho(\Vin^*)/\mumax$.

{\scriptsize\ding{117}}  For $c>0$, the distinct equilibrium $\xia=\transp{(v^*,c^*,q^*)}$
might exist under suitable conditions.
From \eqref{eq:Ic}, $q^*=\mu^{-1}(d)$ if $d<\mumax$.
Additionally, if $\mu(q^*)q^*<\rhomax$ (\textit{i.e.}
$q^*<\qmax = \qmin+\rhomax/\mumax$), then \eqref{eq:Iq} has a unique
solution $v^*=\rho^{-1}(\mu(q^*)q^*)$.
Injecting in (\ref{eq:Iv}) gives $c^*=(\Vin^*-v^*)/q^*$
that is well-defined if $v^*<\Vin^*$.

For convenience, we introduce the function $\psi$ such that
$\psi^{-1}(y)=\rho^{-1}(y\cdot\mu^{-1}(y))$.
Hence, $v^*=\psi^{-1}(d)$.
It can be shown that $\psi(v) = \psimax v/(k_c + v)$ with
$\psimax = \mu(\qmax)=\mumax\rhomax/(\rhomax + \qmin\mumax)$
and $k_c = k_v\qmin\mumax/(\rhomax + \qmin\mumax)$,
$\psi$ is therefore a Monod function.
Consequently, $\psi$ is increasing and bounded above by $\psimax$.
Thus, $\psi^{-1}(d)=v^*<\Vin^* \implies d<\psi(\Vin^*) < \psimax < \mumax$,
this in turn guarantees $q^*<\qmax$.
This establishes $d<\psi(\Vin^*)$ as a necessary and sufficient condition
for the existence of $\xia$ defined at
\begin{equation}\label{eq:Xa-coord}
    v^*=\psi^{-1}(d),~q^*=\mu^{-1}(d),~c^*=\frac{\Vin^*-v^*}{q^*}.
\end{equation}

Although the vitamin feed $\Vin^*$ is meaningful for the algal model
and convenient for defining $\xia$ and its stability,
it carries less significance in the consortium model as it depends
on the parameters $\alpha$ and $\Sin$.
In order to relate the existence of $\xia$ to these parameters,
we inject \eqref{eq:Xa-coord} and \eqref{eq:vin} in the inequality
$v^*<\Vin^*$, which gives
$\psi^{-1}(d)<\alpha\beta\gamma\pp{\Sin-\phi^{-1}(d/(1-\alpha)}$.
Introducing the increasing function $\psi_\alpha$ defined such that
\begin{equation}\label{eq:psi-alpha}
\psi_\alpha^{-1}(y) = \phi^{-1}\pp{\frac{y}{1-\alpha}} + \frac{\psi^{-1}(y)}{\alpha\beta\gamma}
\end{equation}
allows rewriting the last inequality as $\psi_{\alpha}^{-1}(d)<\Sin$.
Subsequently, $\xia$ exists if and only if $d<\psi_\alpha(\Sin)$. 
When $\xia$ exists, it is GAS over $\Omega_A\setminus\cb{c=0}$ and
the washout is unstable. Otherwise, the washout point is GAS over $\Omega_A$.
The stability of these equilibria is extensively studied in \cite{lange_1992}.

\subsubsection{Consortium's steady states}

The algal-bacterial system has up to three equilibrium points:
A washout equilibrium for both species always exists at
$\xo=\transp{(\Sin,0,0,\qmin,0)}$, an algal washout point at
$\xio=\transp{(s^*,e^*,\Vin^*,q_0,0)}$ if
$d<(1-\alpha)\phi(\Sin)$, and an equilibrium with
both species at $\xii=\transp{(s^*,e^*,v^*,q^*,c^*)}$
if $d<\psi(\Vin^*)$.
Stability of these equilibria (Table \ref{tab:eq_stability})
results from the stability of (\ref{eq:bac}-\ref{eq:alg})
and the boundedness of all orbits of the full model
\cite{seibert_1990}.
Stability is not adressed at bifurcation values for its
practical insignificance \cite{harmand_2007}.
From now on, $\xii$
is referred to as the functional equilibrium, denoted $\xeq$,
being the equilibrium of interest with nonzero bacteria and algae.

\begin{table}[H]
    \centering
    \caption{Stability of equilibria over
    $\Omega\setminus\cb{e=0,c=0}$\\
    with $d_1=\psi_\alpha(\Sin)$ and $d_2=(1-\alpha)\phi(\Sin)$}
    \label{tab:eq_stability}
    \begin{tabular}{|G|c|c|c|}
        \hline\rowcolor{verylightgray}
        & $0 < d < d_1$
        & $d_1 < d < d_2$
        & $d_2 < d$\\\hline
        \xo     & unstable & unstable & GAS\\
        \xio    & unstable & GAS & --\\
        \xii    & GAS       & -- & --\\\hline
    \end{tabular}
\end{table}

\section{Optimal Control Problems}\label{pb}

The objective of our study is to maximize the production of
microalgal biomass in a photobioreactor system through two
control variables: the vitamin production activation $\alpha(t)$
and the dilution rate $d(t)$.
Let $\uu(t)=\transp{(\alpha(t),d(t))}$.
The produced microalgae are harvested from the bioreactor output
which is the function $f_0(\xx(t),\uu(t))=d(t)c(t)$.
The system dynamics \eqref{eq:ds}-\eqref{eq:dc} can be expressed
as $\dxt=f(\xx,\uu)$.

The optimal control problem (OCP) for the system
in section \ref{model} is approached in two ways:
the static OCP in which the control inputs are chosen to maximize
the produced microalgae biomass at equilibrium,
and the dynamical OCP where the total harvested microalgae biomass
is optimized over a finite time period.
In both cases, the problem consists of choosing
$\alpha\in\AA$ and $d\in\DD$, where $\UU=\AA\times\DD$
is the relevant set of admissible controls, in order to maximize
the harvested microalgae from the bioreactor.
The two problems are formalized next.

The definition of the dynamic and static OCPs that we propose
in this section closely follows the generic formulation found in
\cite{trelat2015turnpike,caillau2022turnpike,djema2021turnpike}.

\subsection{Static Optimal Control Problem (SOCP)}\label{pb-socp}

The SOCP consists of choosing admissible control inputs
$\uu=\transp{(\alpha,d)}\in\UU$ that maximize
$f_0(\xx,\uu)=d\cdot c$ at the equilibrium over the set
of reachable states $\RR\subset\Omega$, \textit{i.e.},
\begin{equation}\label{eq:socp}\tag{SOCP}
\max\limits_{\substack{(\xx,\uu)\in\RR\times\UU\\ f(\xx,\uu)=0}}
f_0(\xx,\uu).
\end{equation}

\subsection{Dynamic Optimal Control Problem (DOCP)}\label{pb-docp}

We aim to find optimal control function $\uu\in\UU=\AA\times\DD$
to maximize the total production of microalgal biomass over a finite
time horizon $[0,t_f]$, \textit{i.e.},
\begin{equation}\tag{DOCP}
\begin{aligned}
&\max\limits_{\uu\in\UU}
\int_0^{t_f} f_0(\xx(t),\uu(t))\dt \\
&\text{subject to }\dxt(t)=f(\xx(t),\uu(t)), \quad\xx(t)\in\Omega
\end{aligned}
\end{equation}
and verifying the boundary condition $\xx(0)\in\Omega$,
where $\AA$ and $\DD$ are defined here as
\begin{align*}
\AA&=\cb{\alpha\in L^{\infty}([0,t_f]) \vert 0\leq\alpha(t)\leq 1,\forall t\in[0,t_f]},\\
\DD&=\cb{d \in L^{\infty}([0,t_f]) \vert 0 \leq d(t)\leq\dmax,\forall t\in[0,t_f]}. 
\end{align*}

\section{Analysis and numerical solution of SOCP}\label{socp}

As explored in Section \ref{model}, the system has an
asymptotically stable equilibrium $\xeq$ with strictly positive
microalgal biomass if and only if $d<\psi_\alpha(\Sin)$.
Considering $\AA=(0,1)$ and $\DD_\alpha=(0,\psi_\alpha(\Sin))$
for $\alpha\in\AA$, we define
$f_0^*(\uu)=f_0(\xx^*(\uu),\uu)=d\cdot c^*(\alpha,d)$
over $\UU=\AA\times\DD_\alpha$ which relaxes the condition $f(\xx,\uu)=0$.
By injecting the expression of $c^*$ from \eqref{eq:Xb-coord},
\eqref{eq:vin}, \eqref{eq:Xa-coord}, and \eqref{eq:psi-alpha} we obtain
the equivalent problem
\begin{align}
\max\limits_{(\alpha,d)\in\UU} f_0^*(\alpha,d)
&= \frac{d\pp{\Vin^*(\alpha,d) - \psi^{-1}(d)}}{\mu^{-1}(d)}\label{eq:g-a}\\
&= \pp{\Sin - \psi_\alpha^{-1}(d)}
   \pp{\frac{\alpha\beta\gamma d}{\mu^{-1}(d)}}\label{eq:g-d}
\end{align}
that is a constrained maximization problem.

\subsection{Convexity properties of the problem}

A maximization problem is convex if its domain is a convex set and
its objective function $g$ is concave (i.e. $-g$ is convex) 
\cite{boyd_2004}.
The complexity of the expression $f_0^*$ and of the definition
of the domain $\UU$ makes the problem's convexity analysis cumbersome.
However, we can prove weaker properties for
(\ref{eq:g-a}-\ref{eq:g-d}) that still guarantee the existence
of global maxima. In particular, we will rely on the notion of biconvexity~\cite{gorski_2007}.

\begin{remark}\label{monod-convex}
Every Monod function is strictly concave and its inverse is strictly
convex. Both a Monod function and its inverse are strictly positive
and strictly increasing.
\end{remark}

\begin{lemma}\label{lem:a}
$f_0^*$ is strictly concave with respect to $\alpha$.
\end{lemma}
\begin{proof}
From \eqref{eq:g-a}, it is clear that $f_0^*$ is concave
with respect to $\alpha$ if and only if
$\Vin^*(\alpha,d) = \alpha\beta\gamma\pp{\Sin - \phi^{-1}\pp{\frac{d}{1-\alpha}}}$
is concave since $f_0^*$ is affine with respect to $\Vin^*$ with a
positive coefficient $d/\mu^{-1}(d)$.
Moreover, the function $\alpha\mapsto d/(1-\alpha)$ is strictly positive,
increasing, and convex.
Since the composition of an increasing function (here $\phi^{-1}$)
with a strictly convex increasing function is also strictly convex increasing
function \cite{boyd_2004}, it follows that
$h(\alpha)=\Sin-\phi^{-1}(d/(1-\alpha))$ is strictly concave and decreasing.
Finally, $\alpha\mapsto \alpha h(\alpha)$ is strictly concave for $\alpha>0$
because its second derivative $\alpha\mapsto 2h'(\alpha)+\alpha h''(\alpha)$
is strictly negative given $h$ has strictly negative first and second
derivatives. Hence, $\Vin^*(\alpha,d)$ is strictly concave with
respect to $\alpha$ and then so is $f_0^*$.
\end{proof}

\begin{lemma}\label{lem:d}
$f_0^*$ is strictly log-concave w.r.t. $d$.
\end{lemma}
\begin{proof}
Function $\phi_\alpha^{-1}(d)$ defined in \eqref{eq:psi-alpha} is
strictly convex with respect to $d$ for all $\alpha\in\AA$ since $\phi^{-1}$
is increasing and convex, and so is $\psi^{-1}$.
It follows that $\Sin - \phi_\alpha^{-1}(d)$ is strictly concave,
and positive for all $d\in\DD_\alpha$ thus strictly log-concave.
Moreover, $\alpha\beta\gamma d/\mu^{-1}(d)$ is a concave parabola
and strictly positive for $d\in\DD_\alpha$, then it is strictly log-concave.
From \eqref{eq:g-d}, $f_0^*$ is the product of two strictly log-concave
functions, hence it is strictly log-concave w.r.t. $d$ \cite{boyd_2004}.
\end{proof}
\begin{proposition}\label{biconvex}
$f_0^*$ is strictly log-biconcave over $\UU$.
\end{proposition}
\begin{proof}
Since $f_0^*$ is positive and strictly concave with respect to $\alpha$ over $\AA$
(Lemma \ref{lem:a}), it is strictly logarithmically concave w.r.t. $\alpha$.
From Lemma \ref{lem:d}, $f_0^*$ is strictly logarithmically concave w.r.t $d$
over $\DD_\alpha$. Moreover, $\AA$ and $\DD_\alpha$ are open intervals therefore
trivially convex sets. Subsequently, $\UU=\AA\times\DD_\alpha$ is a biconvex
set and $\ln f_0^*$ is a biconcave function.
\end{proof}

\subsection{SOCP solution}

\begin{figure}[t]
    \centering
    \includegraphics[width=\columnwidth]{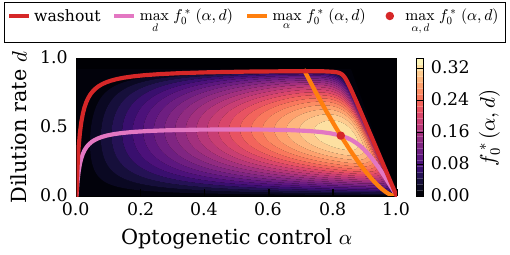}
    \vspace{-1em}
    \caption{SOCP objective function $f_0^*(\alpha,d)$ contours,
    its maxima along $\alpha$ and $d$, and its global maximum.} 
    \label{fig:socp-obj}
\end{figure}

Proposition \ref{biconvex} establishes the logarithmic biconvexity of the
objective function $f_0^*$ over $\UU$, which proves the existence,
although not the uniqueness, of global maximizers \cite{gorski_2007}.
Note that maximizers of $\ln f_0^*(\uu)$ are also maximizers of $f_0^*(\uu)$.
In practice, sequential maximization
(w.r.t. $\alpha$ then w.r.t. $d$) produces
a single global maximum in Figure \ref{fig:socp-obj}
at $\bar\alpha\approx 0.8251$ and $\bar d\approx 0.4409$ per day,
which gives approximately $330.786~m\gL{\cdot}\perday$ of harvested microalgae
for the parameters in Table \ref{tab:jc_params}, with $\Sin = 0.5~\gL$.

\section{DOCP Analysis: Application of the PMP}\label{docp}

In this section, we derive necessary conditions on the optimal controls using the PMP.

\subsection{Hamiltonian and switching functions}
For
$\xx=\transp{(s,e,v,q,c)}\in\Omega$,
$\lmd=\transp{(\ls,\le,\lv,\lq,\lc)}\in\mathbb{R}^5$,
$\uu\in\UU$, and a scalar $\lambda_0$, 
we define the Hamiltonian $H(\xx,\lmd,\lambda_0,\uu)
= \dotp{\lmd, f(\xx,\uu)} + \lambda_0 f_0(\xx,\uu)$, that is,
\begin{equation}\label{eq:H}
\begin{split}
    H =&\lambda_0 d(t)c
    +\ls\bb{- \frac{1}{\gamma}\phi(s)e + d(t)(\Sin - s)}\\
    &+ \le\bb{(1-\alpha(t))\phi(s) - d(t)}e\\
    &+ \lv\bb{\alpha(t)\beta\phi(s)e - \rho(v)c - d(t)v}\\
    &+ \lq\bb{\rho(v) - \mu(q)q}
    + \lc \bb{\mu(q) - d(t)}c.
\end{split}
\end{equation}
Note that this can be written as
\begin{equation}\label{eq:H1}
H = \tilde{H}(\xx,\lmd) + \zeta_\alpha(\xx,\lmd)\alpha(t)
  + \zeta_d(\xx,\lmd,\lambda_0)d(t)
\end{equation}
where $\tilde{H}(\xx,\lmd) =- \frac{\ls}{\gamma}\phi(s)e +\le\phi(s)e -\lv\rho(v)c+\lq\bb{\rho(v)-\mu(q)q} + \lc\mu(q)c$ and
\begin{align}
\zeta_\alpha (\xx,\lmd) =&\pp{\beta\lv -  \le} \phi(s)e,\label{z-a}\\
\zeta_d (\xx,\lmd,\lambda_0) =&~\ls (\Sin-s) - \le e
    - \lv v + (\lambda_0-\lc) c. 
\label{z-d}
\end{align}

\subsection{The co-state dynamics and transversality conditions}
From the PMP, along extremal solutions, there exists an adjoint state $\lmd : [0,t_f] \longrightarrow \mathbb R^5$ absolutely continuous and a $\lambda_0 \geq 0$ such that $(\lmd,\lambda_0)$ is non-trivial, satisfying 
\begin{equation}\label{lambda-dyn-H}
    \dlmd = - \partial H / \partial \xx.
\end{equation}
Next, since the final state $\xx(t_f)$ is free, the transversality conditions are given by
\begin{equation}\label{transv}
\lmd(t_f) = 0.
\end{equation}
These conditions, combined with the co-state dynamics, form a boundary value
problem that must be solved simultaneously with the state equations to
determine the optimal trajectories of the controls $\alpha(t)$ 
and $d(t)$ over $[0,t_f]$.

\subsection{The PMP maximization condition}

The PMP states that for the DOCP, the control strategy maximizes the
Hamiltonian $H$ over $[0,t_f]$.
Thus the PMP maximization condition requires that, for all $t \in [0, t_f]$,
\begin{equation}\label{PMPmax_cond}
H(\xx,\lmd,\lambda_0,\uu) = \max_{\vv\in\UU} H(\xx,\lmd,\lambda_0,\vv).
\end{equation}
Given the form of \eqref{eq:H1}, we establish the following result.

\begin{proposition} For a fixed $t_f>0$, 
the optimal controls $\alpha$ and $d$ satisfy, for almost all $t\in[0,t_f]$,
\begin{equation}\label{a-opt}
\alpha(t)=\begin{cases}
    0 & \text{if }\zeta_\alpha(t) < 0,\\
    1 & \text{if }\zeta_\alpha(t) > 0,\\
    \asing(t) & \text{if }\zeta_\alpha(t) = 0,~t\in[t_1^\alpha,t_2^\alpha],
\end{cases}
\end{equation}
\begin{equation}\label{d-opt}
d(t)=\begin{cases}
    0 & \text{if }\zeta_d(t) < 0,\\
    \dmax & \text{if }\zeta_d(t) > 0,\\
    \dsing(t) & \text{if }\zeta_d(t) = 0,~t\in [t_1^d,t_2^d].
\end{cases}
\end{equation}
\end{proposition}
\begin{proof}
Expressions \eqref{a-opt}-\eqref{d-opt} follow directly from the application
of the PMP maximization condition \eqref{PMPmax_cond} in the light of 
\eqref{eq:H1}, which is affine w.r.t. both
controls $\alpha$ and $d$, such that the signs of the switching functions
$\zeta$ determine the bang controls.
If a switching function is zero over a time interval $[t_1,t_2]$
a singular arc may occur.
The derivation of the analytic expressions of $\asing$ and $\dsing$
(which are both of order 1) is not reported due to
space limitations.
\end{proof}

\subsection{The optimal controls at the final time}
We have the following result.
\begin{proposition}
\label{zeta-tf}
At the final time $t_f$,
   \begin{align}
    \zeta_\alpha(t_f) &= 0\quad\text{and}\quad
    \dot{\zeta}_\alpha(t_f)=0,\label{c1}\\
    \zeta_d(t_f) &= \lambda_0 c(t_f) > 0. \label{c2}
    \end{align}
Moreover, there exists a strictly positive $\varepsilon>0$,
    such that for all $t\in[t_f - \varepsilon, t_f]$,
    $d(t) = \dmax$.
\end{proposition}
\begin{proof}
The expressions in~\eqref{c1} and~\eqref{c2} follow from those of $\zeta_\alpha$ and $\zeta_d$, the costate dynamics~\eqref{lambda-dyn-H}, and the transversality conditions
\eqref{transv}.
The inequality in~\eqref{c2} follows from $\lambda_0>0$ and  $c(t)$ being
positive for every $c(0)>0$.
The second statement follows from \eqref{c2} and the continuity of $c$. 
\end{proof}

Proposition \ref{zeta-tf} asserts that the optimal dilution strategy $d(t)$
ends with a constant \textit{bang} arc $\dmax$. This makes sense from a biotechnological standpoint since
it corresponds to harvesting the remaining algae from
the bioreactor.
The final stretch of $\alpha(t)$
cannot be characterized in the same manner since a singular
arc occurs only if $\zeta_\alpha$ is zero on a \textit{non-empty}
interval, which was not confirmed nor excluded.

\section{Numerical solution of the DOCP and comparison with SOCP}\label{results}

In this section, we employ a direct optimization method that converts
the DOCP into a nonlinear programming
problem (NLP), wherein the state variables as well as the two control
functions $\alpha$ and $d$ are discretized.
We implement the DOCP over a large time interval ($t_f=20$ days) and solve it in {\tt Bocop}~\cite{Bocop} with the following settings: {\tt Gauss II} discretization (implicit Gauss-Legendre 2-stage scheme of order 4); 7000 time steps; tolerance (Ipopt option) $10^{-14}$. 
In the numerical illustration, we utilize the parameter values of Table
\ref{tab:jc_params}, with $\Sin = 0.5~\gL$, and we consider the experimental initial condition
$\xx_{\text{init}}=\transp{(s_0,e_0,v_0,q_0,c_0)}=\transp{(0.1629, 0.0487, 0.0003, 17.7, 0.035)}$ (in the relevant units).

\subsection{The optimal optogenetic control $\alpha(t)$}

The numerically optimized control $\alpha$ (Figure \ref{fig:docp-control}, top) is primarily
\textit{bang-bang}, with short singular arcs between the recurring bangs.
The corresponding switching function $\zeta_\alpha$ is obtained by injecting
the state and costate trajectories into \eqref{z-a}.
The figure shows correspondence between
the sign of the switching function $\zeta_\alpha$ and the optimal $\alpha$, with occurrence of short singular arcs
($\zeta_\alpha=0$ over short time-intervals).
This confirms the properties of 
the control \eqref{a-opt} obtained from the PMP.
We also observe $\zeta_\alpha(t_f)=0$ as per Proposition \ref{zeta-tf}. Concerning the comment on $\alpha$ after the proposition, here a singular arc near $t_f$ is not apparent.

\begin{figure}[t]
    \centering
    \includegraphics[width=\columnwidth]{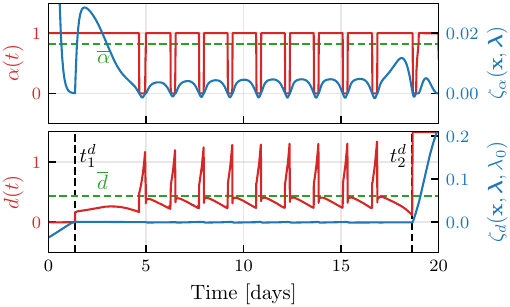}
    \caption{Optimal controls $\alpha(t)$ and $d(t)$ [red] and switching functions $\zeta_\alpha(t)$ and $\zeta_d(t)$ [blue]
    over $t\in[0,t_f]$, along with the optimal static controls $\bar{\alpha}$
    and $\bar{d}$ [dashed green].}
    \label{fig:docp-control}
\end{figure}

\subsection{The optimal dilution rate control $d(t)$}

The numerically optimized control $d$ (Figure \ref{fig:docp-control}, bottom) has three major
phases: a minimal bang arc $d(t)=0$ for all $t\in[0,t_1^d)$, then a singular arc
$\dsing(t)$ for $t\in[t_1^d,t_2^d]$, followed by a maximal bang arc $d(t)=\dmax$
for $t\in(t_2^d,t_f]$ as per Proposition \ref{zeta-tf}.
This illustrates the properties of the control \eqref{d-opt} obtained from the PMP.
The singular arc $\dsing(t)$ exhibits a cyclic-like behavior that appears to
be paced by the optimal control $\alpha(t)$.
This behaviour will be object of further investigation.

\subsection{SOCP and DOCP comparison}

\begin{figure}[t]
    \centering
    \includegraphics[width=\columnwidth]{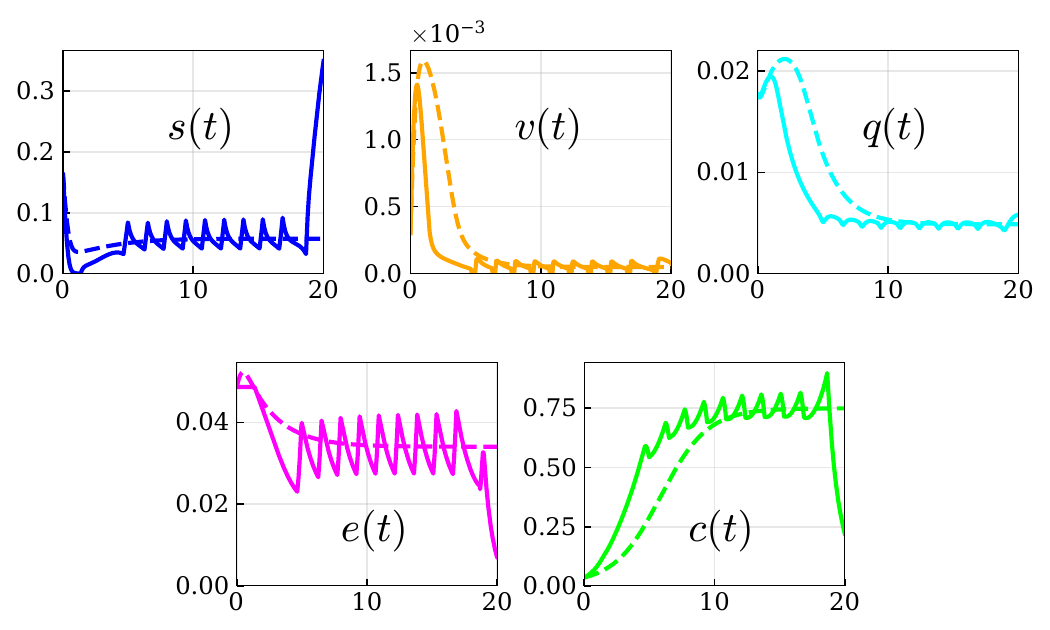}
    \caption{The optimal state trajectories plotted against time
    over 20 days, under the optimal dynamic control $u(t)$
    [solid lines] and under the static optimal control
    $\bar{u}$ [dashed lines].
    Concentrations are given in $\gL$ and the internal cell quota is given in $g{\cdot}g^{-1}$}
    \label{fig:ocp-state}
\end{figure}

The static optimal control law $\bar{\uu}$ obtained in Section \ref{socp} for the equilibrium state $\xx^*$ is compared with the dynamic optimal control law solving the DOCP with initial condition $\xx(0)=\xx_{\text{init}}$ (see Figure \ref{fig:docp-control} for a comparison of the control laws). First, we apply both laws to the system starting from initial condition $\xx(0)=\xx_{\text{init}}$. 
The resulting state trajectories 
are illustrated in Figure \ref{fig:ocp-state} for $t\in[0,t_f]$ and $t_f=20$ days.
In a second test, we apply both laws to the system starting from initial condition $\xx(0)=\xx^*$.
The performance of the two strategies in terms of total harvested microalgae over the period $[0,t_f]$ is compared in Table \ref{tab:x0}.
When the system starts from $\xx(0)=\xx_{\text{init}}$ (first row of the table), the dynamic strategy is optimal and it improves upon the static one by approximately 16.85\%, which
is remarkable for biotechnological applications.
The dynamic strategy outperforms the static one (by approximately 5.88\%) even when the system starts from $\xx(0)=\xx^*$ (second row), the scenario for which the SOCP was solved. 
We draw the conclusion that the cyclic-like regime of the DOCP solution significantly improves the objective criterion, a phenomenon known as \textit{overyielding}
(see~\cite{bayen2020optimal,bayen2020improvement} and references therein).

\begin{table}[h!]
\centering
\caption{Total harvested microalgae over $t_f=20$ days.}
\label{tab:x0}
\begin{tabular}{ccc}\toprule
$\xx(0)$& static optimal control & dynamic optimal control\\\midrule
$\xx_{\text{init}}$ & $4.665953~(g/L)$ & $5.45218~(g/L)$\\ [0.5em]
$\xeq$ & $6.615718~(g/L)$ & $7.005~(g/L)$\\
\bottomrule
\end{tabular}
\end{table}

\section{Conclusions}\label{conclusion}
In this study, we discussed static and dynamical optimal control of a microbial consortium model,
with the objective to enhance  productivity of algal biomass in the interest of biotechnological applications.
In particular, our investigation reveals that the dynamic control involves
\textit{bang-bang} and singular phases, leading to a significant improvement
of the control performance achieved with a static optimal solution at equilibrium.
In our future work, we aim to explore the design of robust feedback controls for practical
implementation in bioreactors.

\section*{ACKNOWLEDGMENT}

We thank Juan-Carlos Arceo-Luzanilla for contributing to establish parameter values, Antoine Sciandra for experimental data, and Inria project-team MICROCOSME for modelling discussions.

\bibliography{ref.bib}
\bibliographystyle{IEEEtran}

\end{document}